\documentclass[11pt, a4paper]{amsart}
\usepackage{amssymb, array, amsmath, amscd, pdfpages, enumerate, amsthm, fixltx2e, setspace, hyperref}

\DeclareMathOperator*{\dist}{dist}

\DeclareMathOperator*{\sgn}{sgn}

\newcommand{\e}{\mathrm{e}}
\newcommand{\ii}{\mathrm{i}}
\newcommand{\dd}{\mathrm{d}}

\newcommand{\mlog}{m_\mathrm{log}}

\newcommand{\Lloc}{L^1_\mathrm{loc}}

\newcommand{\B}{\mathcal{B}}

\newcommand{\T}{\mathbb{T}}
\newcommand{\RR}{\mathbb{R}}
\newcommand{\CC}{\mathbb{C}}

\newcommand{\ZZ}{\mathbb{Z}}

\newcommand{\EE}{\mathbb{E}}

\newtheorem{thm}{Theorem}[section]

\theoremstyle{definition}
\newtheorem{rem}[thm]{Remark}
\newtheorem{ex}[thm]{Example}

\numberwithin{equation}{section}

\begin{document}

\title[A quantified Tauberian theorem for sequences]{ A quantified Tauberian theorem for~sequences}

\author{David Seifert}
\address{St John's College, St Giles, Oxford\;\;OX1 3JP, United Kingdom}
\email{david.seifert@sjc.ox.ac.uk}
\thanks{The author is grateful to Professor Yuri Tomilov for several helpful comments on an earlier version of this manuscript.}

\begin{abstract}
The main result of this paper is a quantified version of Ingham's Tauberian theorem for bounded vector-valued  sequences rather than functions. It gives an estimate on the rate of decay of such a sequence in terms of the behaviour of a certain boundary function, with the quality of the estimate depending on the degree of smoothness this boundary function is assumed to possess.  The result is then used to give a new proof of the quantified Katznelson-Tzafriri theorem recently obtained in \cite{Se2}.
 \end{abstract}


\maketitle

\section{Introduction}\label{intro} 

One of the cornerstones in the asymptotic theory of operators is the Katznelson-Tzafriri theorem \cite[Theorem~1]{KT86}, which states the following.

\begin{thm}\label{KT_thm}
Let $X$ be a complex Banach space and suppose that $T\in\B(X)$ is power-bounded. Then 
\begin{equation}\label{eq:KT}
\lim_{n\to\infty}\|T^n(I-T)\|=0
\end{equation}
if any only if $\sigma(T)\cap\T\subset\{1\}$.
\end{thm}

Here $\B(X)$ denotes the algebra of bounded linear operators on a complex Banach space $X$, $\sigma(T)$ denotes the {spectrum} of the operator $T\in\B(X)$, and an operator $T\in\B(X)$ is said to be {power-bounded} if $\sup_{n\ge0}\|T^n\|<\infty$. Moreover, $\T$ stands for the unit circle $\{\lambda\in\CC:|\lambda|=1\}.$ 

Limits of the type appearing in \eqref{eq:KT} play an important role for instance in the theory of iterative methods (see \cite{Ne93}), so it is natural to ask at what \emph{speed} convergence takes place. If $\sigma(T)\cap\T=\emptyset$ the decay is at least exponential, with the rate determined by the spectral radius of $T$,  so the real interest is in the non-trivial case where $\sigma(T)\cap\T=\{1\}$.  Given a continuous non-increasing function $m:(0,\pi]\to[1,\infty)$ such that $\|R(\e^{\ii\theta},T)\|\leq m(|\theta|)$ for $0<|\theta|\leq\pi$, it is shown in \cite[Theorem~2.11]{Se2} that, for any $c\in(0,1)$,  
$$\|T^n(I-T)\|=O\big(\mlog^{-1}(cn)\big),\quad n\to\infty,$$
where $\mlog^{-1}$ is the inverse function of the map $\mlog$ defined by
\begin{equation}\label{mlog}
\mlog(\varepsilon)=m(\varepsilon)\log\left(1+\frac{m(\varepsilon)}{\varepsilon}\right),\quad0<\varepsilon\leq\pi,
\end{equation}
and where the statement $x_n=O(y_n)$, $n\to\infty$, for two sequences $(x_n)$, $(y_n)$ of non-negative terms, means that there exists a constant $C>0$ such that $x_n\le C y_n$ for all sufficiently large $n\ge0$. Moreover, this result is optimal in an important special case; see Remark~\ref*{KT_rem}\eqref{opt_rem} below. 

The main new result of this paper, Theorem~\ref{Ing}, is a Tauberian theorem for  sequences. The result is formulated for bounded vector-valued sequences but, to the knowledge of the author, is new even in the scalar-valued case. It can be viewed as a discrete analogue of Ingham's classical Tauberian theorem for functions; however, it includes an estimate on the rate of decay. This is achieved by adapting a new technique developed recently in the setting of $C_0$-semigroups in \cite{CS} and going back to \cite{BCT}.  The result is then used, in Theorem~\ref{KT_quant}, to give a new proof of the quantified version of Theorem~\ref{KT_thm} discussed above. For further related results in the general area may be found in \cite{AOFR87},  \cite{Ba94b}, \cite{Du08a}, \cite{KT86}, \cite{Le14}, \cite{Ne11}, \cite{Ra88}, \cite{Rie16} and the references they contain.

\section{Main results}\label{ingham}

\subsection{Preliminaries}

Let $X$ be a complex Banach space and write $C_0(-\pi,\pi)$ for the set of  continuous functions $\psi:[-\pi,\pi]\to\CC$ which vanish in a neighbourhood of zero and satisfy $\psi(-\pi)=\psi(\pi)$. Further let  $\Lloc(\T\backslash\{1\};X)$ denote the set of  functions $F:\T\backslash\{1\}\to X$ such that the map $\theta\mapsto \psi(\theta)F(\e^{\ii\theta})$, interpreted as taking the value zero when $\psi$ does, lies in $L^1(-\pi,\pi;X)$ for all $\psi\in C_0(-\pi,\pi)$. Let $\EE=\{\lambda\in\CC:|\lambda|>1\}$, the exterior of the closed unit disc. Given a holomorphic function $G:\EE\to X$  and given $F\in \Lloc(\T\backslash\{1\};X)$, $F$ will be said to be a \emph{boundary function} for $G$ if 
\begin{equation}\label{extension}
\lim_{r\to1+}\int_{-\pi}^\pi \psi(\theta)G\big(r\e^{\ii\theta}\big)\,\dd\theta=\int_{-\pi}^\pi \psi(\theta)F\big(\e^{\ii\theta}\big)\,\dd\theta
\end{equation}
for all $\psi\in C_0(-\pi,\pi)$. For $k\ge1$, let $C^k(\T\backslash\{1\};X)$ denote the set of functions $F:\T\backslash\{1\}\to X$ which are $k$-times continuously differentiable, with $\T\backslash\{1\}$  viewed as a one-dimensional manifold, and let $C^\infty(\T\backslash\{1\};X)=\bigcap_{k\ge1}C^k(\T\backslash\{1\};X)$. 

\subsection{A quantified Tauberian theorem}

Theorem~\ref{Ing} below is the main result of this paper and can be viewed as a discrete analogue of Ingham's Tauberian theorem for functions; see \cite{Ing33} and also  \cite{Ka34}. In the statement of the result,  given $x\in\ell^\infty(\ZZ_+;X)$, $G_x:\EE\to X$ denotes the holomorphic function given by 
$$G_x(\lambda)=\sum_{n\geq0}\frac{x_n}{\lambda^{n+1}},\quad |\lambda|>1.$$
The theorem shows that if $x\in\ell^\infty(\ZZ_+;X)$ has uniformly bounded partial sums and if $G_x$ possesses a boundary function $F_x$, then $x\in c_0(\ZZ_+;X)$. Moreover, the result gives an estimate on the rate of decay of $\|x_n\|$ as $n\to\infty$, the quality of which depends on the smoothness and the rate of growth near the point $1$ of $F_x$. The proof uses a technique which goes back to \cite{BCT} and has been extended recently in \cite{CS}. One advantage of this approach over the contour integral method used to obtain \cite[Theorem~2.11]{Se2} is that it extends to the case in which $F_x$ is only finitely often continuously differentiable. Given a continuous non-increasing function $m:(0,\pi]\to[1,\infty)$ and $k\ge1$, define the function $m_k:(0,\pi]\to(0,\infty)$ by
\begin{equation}\label{m_k}
m_k(\varepsilon)=m(\varepsilon)\left(\frac{m(\varepsilon)}{\varepsilon}\right)^{1/k},
\end{equation}
noting that, for each $k\ge1$, $m_k$ maps bijectively onto its range.

\begin{thm}\label{Ing}
Let $X$ be a complex Banach space and let $x\in \ell^\infty(\ZZ_+;X)$ be such that
\begin{equation}\label{bdd}
\sup_{n\geq0}\bigg\|\sum_{k=0}^n x_k\bigg\|<\infty.
\end{equation}
If $G_x$ admits a boundary function $F_x\in \smash{\Lloc}(\T\backslash\{1\};X)$, then $x\in c_0(\ZZ_+;X)$.  

Moreover, given a continuous non-increasing function $m:(0,\pi]\to[1,\infty)$,  the following hold.
\begin{enumerate}[(a)]
\item\label{Ck_case} Suppose that $F_x\in C^k(\T\backslash\{1\};X)$ for some  $k\ge1$ and that 
\begin{equation}\label{Ck_dom_fun}
\| F_x^{(j)}(\e^{\ii\theta})\|\le C |\theta|^{\ell-j}m(|\theta|)^{\ell+1},\quad 0<|\theta|\le\pi,\;0\le j\le\ell\le k,
\end{equation}
for some constant $C>0$. Then, for any $c>0$,
 \begin{equation}\label{bound_Ck}
\|x_n\|=O\left(m_k^{-1}\big(cn)\right),\quad n\to\infty,
\end{equation}
where $m_k^{-1}$ is the inverse function of the map $m_k$ defined in \eqref{m_k}.
\item\label{hol_case} Suppose that $F_x\in C^\infty(\T\backslash\{1\};X)$ and that 
\begin{equation}\label{dom_fun}
\| F_x^{(j)}(\e^{\ii\theta})\|\le C j!|\theta| m(|\theta|)^{j+1},\quad 0<|\theta|\le\pi,\;j\ge0,
\end{equation}
for some constant $C>0$. Then, for any $c\in(0,1)$,
\begin{equation}\label{bound}
\|x_n\|=O\left(\mlog^{-1}(cn)+\frac{1}{n}\right),\quad n\to\infty,
\end{equation}
where $\mlog^{-1}$ is the inverse function of the map $\mlog$ defined in \eqref{mlog}.
\end{enumerate}
\end{thm}

\begin{rem}\label{rem0}
\begin{enumerate}[(a)]
\item Neither condition \eqref{bdd} nor the assumption that $G_x$ admits a boundary function can be dropped, even in the scalar-valued case, as can be seen by considering the sequences $x=(1,1,1,\dotsc)$ and $x=(+1, -1,+ 1, -1,\dotsc)$, respectively.
\item Note that if $m(\varepsilon)\ge c/\varepsilon$ for some $c>0$, then \eqref{Ck_dom_fun} is satisfied if 
$$\| F_x^{(j)}(\e^{\ii\theta})\|\le C m(|\theta|)^{j+1}\quad0<|\theta|\le\pi,\;0\le j\le k,$$
 for some constant $C>0$.
\item  
Suppose that $m:(0,\pi]\to[1,\infty)$ is as  in Theorem~\ref{Ing}. If $G_x$ has a holomorphic extension, denoted also by $G_x$, to a region containing 
$$\Omega_{m,\theta}= \left\{\lambda\in\CC:|\lambda-\e^{\ii\theta}|\le\frac{1}{m(|\theta|)}\right\}$$
for  $0<|\theta|\le\pi$ and if 
\begin{equation*}\label{O_bd}
\|G_x(\lambda)\|\le C|\theta|m(|\theta|),\quad \lambda\in\Omega_{m,\theta},\;0<|\theta|\le\pi,
\end{equation*}
for some constant $C>0$, then a simple estimate using Cauchy's integral formula shows that \eqref{dom_fun} holds for the restriction $F_x$ of $G_x$ to $\T\backslash\{1\}$. This is analogous to the results for Laplace transforms in \cite{BD08} and \cite{Ma11}. Conversely, if $F_x\in C^\infty(\T\backslash\{1\};X)$ and \eqref{dom_fun} holds, then $F_x$ extends holomorphically to the region $\Omega_m$ given by 
\begin{equation*}\label{Om}
\Omega_m= \left\{\lambda\in\CC:|\lambda-\e^{\ii\theta}|<\frac{1}{m(|\theta|)},\;0<|\theta|\le\pi\right\}.
\end{equation*}
Furthermore, if  $G:\EE\to X$ is a holomorphic function which admits a boundary function $F_x\in C^\infty(\T\backslash\{1\};X)$ satisfying \eqref{dom_fun}, then $G$  has a holomorphic extension which agrees with that of $F_x$ on $\Omega_m$. This follows from   a standard Cayley transform argument combined with the `edge-of-the-wedge theorem'; see for instance \cite[\S2 Theorem~B]{Ru71}.
\end{enumerate}
\end{rem}

\begin{proof}[Proof of Theorem~\ref{Ing}]
Let $\psi:[-\pi,\pi]\to\RR$ be a smooth function such that  $\psi(\theta)=0$ for $|\theta|\le1$, $0\le\psi(\theta)\le1$ for $1\le|\theta|\le2$ and $\psi(\theta)=1$ for $2\le|\theta|\le\pi.$ For $\varepsilon\in(0,\pi/2]$, let $\psi_\varepsilon, \varphi_\varepsilon:[-\pi,\pi]\to\RR$ be given by $\psi_\varepsilon(\theta)=\psi(\theta/\varepsilon)$ and  $\varphi_\varepsilon(\theta)=1-\psi_\varepsilon(\theta),$ $-\pi\le\theta\le\pi$. Moreover, for $n\in\ZZ$, let
$$y^\varepsilon_n=\frac{1}{2\pi}\int_{-\pi}^\pi \e^{\ii n\theta}\psi_\varepsilon(\theta)\,\dd\theta\quad \mbox{and}\quad z^\varepsilon_n=\frac{1}{2\pi}\int_{-\pi}^\pi \e^{\ii n\theta}\varphi_\varepsilon(\theta)\,\dd\theta.$$
Then $y^\varepsilon_0=1-z_0^\varepsilon$ and $y^\varepsilon_n=-z_n^\varepsilon$ for $n\ne0$, and a simple calculation using integration by parts shows that $y^\varepsilon,z^\varepsilon\in\ell^1(\ZZ)$. Let $x^\varepsilon\in\ell^\infty(\ZZ;X)$ be given by $x^{\varepsilon}=x*y^{\varepsilon}$, so that $\smash{x^{\varepsilon}_n=\sum_{j\geq0}x_j y^{\varepsilon}_{n-j}}$ for $n\in\ZZ.$ Then, setting $\smash{s_n=\sum_{j=0}^n x_j}$ for $n\ge0$,
\begin{equation}\label{x-z}
x_n-x_n^\varepsilon=(x*z^\varepsilon)_n=\sum_{j\ge0}s_j\big(z_{n-j}^\varepsilon-z_{n-j-1}^\varepsilon\big),\quad n\ge0.
\end{equation}
Since $\varphi_\varepsilon(\theta)=0$ for $2\varepsilon\le|\theta|\le\pi$,
\begin{equation}\label{small}
|z_n^\varepsilon-z_{n-1}^\varepsilon|=\left|\frac{1}{2\pi}\int_{-\pi}^\pi \e^{\ii n\theta}\big(1-\e^{-\ii\theta}\big)\varphi_\varepsilon(\theta)\,\dd\theta\right|\lesssim \int_{-2\varepsilon}^{2\varepsilon}|\theta|\,\dd\theta\lesssim\varepsilon^2
\end{equation}
for all $n\in\ZZ$. Here  and in what follows the statement $p\lesssim q$ for  real-valued quantities $p$ and $q$  means that $p\leq Cq$ for some number $C>0$ which is independent of all the parameters that are free to vary, in this case of $\varepsilon$ and $n$. Similarly, for $n\ne0$, integrating by parts twice gives
\begin{equation}\label{big}
|z_n^\varepsilon-z_{n-1}^\varepsilon|=\left|\frac{1}{2\pi n^2}\int_{-\pi}^\pi \e^{\ii n\theta}\frac{\dd^2}{\dd\theta^2}\Big(\big(1-\e^{-\ii\theta}\big)\varphi_\varepsilon(\theta)\Big)\,\dd\theta\right|\lesssim \frac{1}{n^2}.
\end{equation}
For $n\ge0$ and $\varepsilon\in(0,\pi/2]$, let $\smash{P_{n,\varepsilon}=\{j\ge0:|j-n|\le\frac{1}{\varepsilon}}\}$ and $Q_{n,\varepsilon}=\{j\ge0:|j-n|>\smash{\frac{1}{\varepsilon}}\}$. Using \eqref{small} and \eqref{big} in \eqref{x-z}, together with the fact that $s\in\ell^\infty(\ZZ_+;X)$ by assumption \eqref{bdd}, it follows that 
\begin{equation}\label{x-xe}
\|x_n-x_n^\varepsilon\|\lesssim \sum_{j\in P_{n,\varepsilon}}\varepsilon^2+\sum_{j\in Q_{n,\varepsilon}}\frac{1}{(n-j)^2} \lesssim\varepsilon,\quad n\ge0.
\end{equation} 
Now, by the dominated convergence theorem, Fubini's theorem and \eqref{extension}, 
\begin{equation}\label{RL}
\begin{aligned}
x^{\varepsilon}_n&=\lim_{r\to1+}\sum_{j\geq0}\frac{x_j}{r^{j+1}}y^{\varepsilon}_{n-j}\\
&=\lim_{r\to1+}\frac{1}{2\pi}\sum_{j\geq0}\int_{-\pi}^\pi \frac{x_j}{r^{j+1}}\e^{\ii(n-j)\theta}\psi_\varepsilon(\theta)\,\dd\theta   \\
&=\lim_{r\to1+}\frac{1}{2\pi}\int_{-\pi}^\pi\e^{\ii(n+1)\theta}\psi_\varepsilon(\theta)G_x\big(r\e^{\ii\theta}\big)\,\dd\theta\\
&=\frac{1}{2\pi}\int_{-\pi}^\pi\e^{\ii(n+1)\theta}\psi_\varepsilon(\theta)F_x\big(\e^{\ii\theta}\big)\,\dd\theta
\end{aligned}
\end{equation}
for all $n\in\ZZ$ and  $\varepsilon\in(0,\pi/2]$. Hence $x^\varepsilon\in c_0(\ZZ;X)$ for each $\varepsilon\in(0,\pi/2]$ by the Riemann-Lebesgue lemma, and it follows from \eqref{x-xe} that $x\in c_0(\ZZ_+;X)$.

Suppose  $F_x\in \smash{C^k}(\T\backslash\{1\};X)$ for some  $k\ge1$. Integrating by parts $k$ times in \eqref{RL} and estimating crudely by means of \eqref{Ck_dom_fun} gives 
\begin{equation}\label{xe}
\|x_n^\varepsilon\|\lesssim \frac{1}{n^k}\sum_{j=0}^k m(\varepsilon)^{j+1}\lesssim \frac{m(\varepsilon)^{k+1}}{n^k}
\end{equation}
for all $n\ge1$ and all  $\varepsilon\in(0,\pi/2]$. Given $c>0$ and $n\ge1$ sufficiently large, let $\varepsilon_n\in (0,\pi/2]$ be given by $\smash{\varepsilon_n=m_k^{-1}(cn)}$. The estimate \eqref{bound_Ck} follows from \eqref{x-xe} and \eqref{xe} on setting on setting $\varepsilon=\varepsilon_{n}$ for sufficiently large $n\ge1$.

Now suppose that $F_x\in C^\infty(\T\backslash\{1\};X)$. In order to obtain the estimate \eqref{bound}, it is necessary to make explicit choices of the functions $\psi_\varepsilon,\varphi_\varepsilon:[-\pi,\pi]\to\RR$ and hence of the sequences $y^\varepsilon,z^\varepsilon\in\ell^1(\ZZ)$ of their Fourier coefficients. Thus, given $\varepsilon\in(0,\pi/2]$,  let $y^\varepsilon\in \ell^1(\ZZ)$ be given by $\smash{y_0^\varepsilon=1-\frac{3\varepsilon}{2\pi}}$ and 
$$y^{\varepsilon}_n=\frac{\cos(2n\varepsilon)-\cos (n\varepsilon)}{\varepsilon\pi n^2},\quad n\ne0,$$
and define $z^\varepsilon\in \ell^1(\ZZ)$  by $z^\varepsilon_0=1-y^\varepsilon_0$ and $z^\varepsilon_n=-y^\varepsilon_n$ for $n\ne0$. Moreover, let $x^{\varepsilon}\in\ell^\infty(\ZZ;X)$ be given by $x^{\varepsilon}=x*y^{\varepsilon}$, as before. Then the function
$$\psi_\varepsilon(\theta)=\sum_{n\in\ZZ} \frac{y_n^\varepsilon}{\e^{\ii n\theta}},\quad -\pi\le\theta\le\pi,$$
satisfies $\psi_\varepsilon(\theta)=0$ for $|\theta|\leq\varepsilon$, $\psi_\varepsilon(\theta)=\varepsilon^{-1}|\theta|-1$ for $\varepsilon\leq |\theta|\leq2\varepsilon$, and $\psi_\varepsilon(\theta)=1$ for $2\varepsilon\leq|\theta|\leq\pi$. Now \eqref{x-z} still holds but the above method for estimating $|z_n^\varepsilon-z_{n-1}^\varepsilon|$, $n\in\ZZ$, is no longer applicable since $\varphi_\varepsilon$ is not differentiable. Instead, consider the function  $\phi:\RR\to\RR$ given by $\phi(0)=0$ and 
$$\phi(t)=\frac{2}{\pi}\left(\frac{\cos(2t)-\cos(t)}{t^3}+\frac{\sin(2t)-\frac{1}{2}\sin(t)}{t^2}\right),\quad t\ne0.$$
Then $\phi\in L^1(\RR)$ and 
$$z_{n}^\varepsilon-z_{n-1}^\varepsilon=\varepsilon\int_{\varepsilon(n-1)}^{\varepsilon n}\phi(t)\,\dd t,\quad n\in\ZZ,$$
and it follows from \eqref{bdd} and \eqref{x-z} that
\begin{equation}\label{difference}
\|x_n-x_n^\varepsilon\|\le\varepsilon\sum_{j\ge0}\int_{\varepsilon(n-j-1)}^{\varepsilon (n-j)}|\phi(t)|\,\dd t
\lesssim \varepsilon,\quad n\ge0.
\end{equation}
Now, by the same argument as in \eqref{RL},
$$x_n^\varepsilon=\frac{1}{2\pi}\int_{-\pi}^\pi\e^{\ii(n+1)\theta}\psi_\varepsilon(\theta)F_x\big(\e^{\ii\theta}\big)\,\dd\theta,\quad n\in\ZZ.$$
Integrating by parts $k\ge1$ times gives
$$\begin{aligned}
x^\varepsilon_n&=A_{n,k}(-1)^k\int_{\varepsilon\leq|\theta|\leq\pi}\e^{\ii(n+k+1)\theta}\psi_\varepsilon(\theta)F_x^{(k)}\big(\e^{\ii\theta}\big)\,\dd\theta\\
&\qquad+B_{n,k}\frac{(-1)^{k}}{2\pi\varepsilon\ii}\int_{\varepsilon\leq|\theta|\leq2\varepsilon}\e^{\ii(n+k)\theta}F_x^{(k-1)}\big(\e^{\ii\theta}\big)\sgn\theta\,\dd\theta\\
&\qquad +\sum_{j=0}^{k-2}C_{n,j}\frac{(-1)^{j}}{2\pi\varepsilon}\left[\e^{\ii(n+j+1)\theta}F_x^{(j)}\big(\e^{\ii\theta}\big)+\e^{-\ii(n+j+1)\theta}F_x^{(j)}\big(\e^{-\ii\theta}\big)\right]_\varepsilon^{2\varepsilon}
\end{aligned}$$
for all $n\geq0$, where 
\begin{gather*} 
A_{n,k}=\frac{n!}{(n+k)!}, \quad B_{n,k}=\frac{n!}{(n+k-1)!}\sum_{j=1}^k\frac{1}{n+j}\\
\mbox{and} \quad C_{n,j}=\frac{n!}{(n+j+1)!}\sum_{\ell=1}^{j+1}\frac{1}{n+\ell}
\end{gather*}
for $0\leq j\leq k-2$. Now $A_{n,k}\le n^{-k}$, $B_{n,k}\le kn^{-k}$ and $C_{n,j}\le(j+1)n^{-(j+2)}$ for $n\ge1$. Thus \eqref{dom_fun} gives
\begin{equation}\label{est}
\|x_n^\varepsilon\|\lesssim k! \frac{ m(\varepsilon)^{k+1}}{n^k}+\frac{m(\varepsilon)}{n^2}\sum_{j=0}^{k-2}(j+1)!\frac{ m(\varepsilon)^j}{n^j}, \quad n,k\ge1,
\end{equation}
for  $\varepsilon\in(0,\pi/2]$. Denote the two terms on the right-hand side of \eqref{est} by $\smash{D^\varepsilon_{n,k}}$ and $\smash{E^\varepsilon_{n,k}}$, respectively.  For $c\in (0,1)$, Stirling's formula implies that $\smash{k!\lesssim (k/c\e )^k}$ for all $k\geq0$ and hence 
$$D^\varepsilon_{n,k}\lesssim m(\varepsilon)\left(\frac{k m(\varepsilon)}{c\e n}\right)^k,\quad n,k\ge1.$$
Let $k_{\varepsilon,n}=\lfloor cn/m(\varepsilon)\rfloor$. Then 
\begin{equation}\label{D_est}
D^\varepsilon_{n,k_{\varepsilon,n}}\lesssim m(\varepsilon)\exp\left(-\frac{cn}{m(\varepsilon)}\right)
\end{equation}
for all  $\varepsilon\in(0,\pi/2]$ and all $n\ge1$ such that $k_{\varepsilon,n}\ge1$. Moreover, for such values of $\varepsilon$ and $n$, the choice of $k_{\varepsilon,n}$ ensures that  
\begin{equation}\label{E_est}
E^\varepsilon_{n,k_{\varepsilon,n}}\leq \frac{m(\varepsilon)}{n^2}\sum_{j=0}^{k-2}c^j\lesssim \frac{m(\varepsilon)}{n^2}.
\end{equation}
 Thus setting $k=k_{\varepsilon,n}$ in \eqref{est} and using \eqref{D_est} and \eqref{E_est} gives
\begin{equation}\label{x_est}
\|x_n^\varepsilon\|\lesssim m(\varepsilon)\exp\left(-\frac{cn}{m(\varepsilon)}\right)+\frac{m(\varepsilon)}{n^2}
\end{equation}
for all $\varepsilon\in(0,\pi/2]$ and $n\ge1$ as above. Let $\varepsilon_n=\smash{\mlog^{-1}(cn)}$ for $n\ge1$ sufficiently large to ensure that $k_{\varepsilon_n,n}\ge1$. For such values of $n$,
$$m(\varepsilon_n)\exp\left(-\frac{cn}{m(\varepsilon_n)}\right)\lesssim\varepsilon_n\quad\mbox{and}\quad \frac{m(\varepsilon_n)}{n^2}\lesssim\frac{1}{n},$$
so  \eqref{bound} follows from \eqref{difference} and \eqref{x_est} on setting $\varepsilon=\varepsilon_n$.
\end{proof}

\begin{rem}\label{rem1}
\begin{enumerate}[(a)]
\item The choice of $k_{\varepsilon,n}$  before equation \eqref{D_est} is motivated by the fact that, given any constant $C>0$, the function $t\mapsto (Ct)^t$, defined on $(0,\infty)$, attains its global minimum at $t=(C\e)^{-1}$.
\item  Theorem~\ref{Ing} can be extended  to the case of a finite number of singularities on the unit circle; see also \cite{Ma11} and \cite{Se2}.
\end{enumerate}
\end{rem}

The following example illustrates that the estimates in Theorem~\ref{Ing} generally improve with the smoothness of the boundary function if $m(\varepsilon)$ grows moderately fast as $\varepsilon\to0+$, but that the quality of these estimates can be independent of the degree of smoothness if this blow-up is very rapid.

\begin{ex}\label{poly_ex}
In Theorem~\ref{Ing}, consider the function $m:(0,\pi]\to[1,\infty)$ given by $m(\varepsilon)=(\pi/\varepsilon)^{\alpha}$, where $\alpha\ge1$. If $F_x\in C^k(\T\backslash\{1\};X)$ for some $k\ge1$ and \eqref{Ck_dom_fun} holds, then \eqref{bound_Ck} gives
$$\|x_n\|=O\left(n^{-\frac{k}{\alpha(k+1)+1}}\right),\quad n\to\infty,$$
and, if $F_x\in C^\infty(\T\backslash\{1\};X)$ and \eqref{dom_fun} holds, \eqref{bound} becomes
\begin{equation}\label{u_bd}
\|x_n\|=O\Bigg(\bigg(\frac{\log n}{n}\bigg)^{\frac{1}{\alpha}}\Bigg),\quad n\to\infty.
\end{equation}
Thus the estimate improves with the smoothness of $F_x$. By contrast, if the assumptions of Theorem~\ref{Ing}  are satisfied for  $m(\varepsilon)=\exp(\varepsilon^{-\alpha})$ with  $\alpha>0$, then \eqref{bound_Ck} for any $k\ge1$ and \eqref{bound} all become 
$$\|x_n\|=O\left((\log n)^{-\frac{1}{\alpha}}\right),\quad n\to\infty,$$
so in this case the quality of the estimate is unaffected by the smoothness of the boundary function.
\end{ex}

\subsection{The quantified Katznelson-Tzafriri theorem}

The purpose of this final section is to deduce from Theorem~\ref{Ing} the quantified version of Theorem~\ref{KT_thm}. This result was first obtained in \cite{Se2} by means of a contour integral argument adapted from \cite{BD08} and \cite{Ma11}. Some closely related results may be found for instance in  \cite{Du08a}, \cite{Le14} and \cite{Ne11}. 

\begin{thm}\label{KT_quant}
Let $X$ be a complex Banach space and let $T\in \B(X)$ be a power-bounded operator such that $\sigma(T)\cap\T=\{1\}$.  Suppose there exists  a continuous non-increasing function $m:(0,\pi]\to[1,\infty)$ such that  
$$\|R(\e^{\ii\theta},T)\|\le  m(|\theta|),\quad 0<|\theta|\le\pi.$$
Then, for any  $c\in(0,1)$, 
 \begin{equation}\label{KT}
 \|T^n(I-T)\|=O\big(  m_{\log}^{-1}(c n)\big),\quad n\to\infty,
 \end{equation}
where $\mlog^{-1}$ is the inverse function of the map $\mlog$ defined in \eqref{mlog}.
 \end{thm}
 
\begin{proof}
The result follows from part (b) of Theorem~\ref{Ing} applied, with $X$ replaced by $\B(X)$, to the sequence $x$ whose $n$-th term is $x_n=T^n(I-T)$, $n\ge0$. Indeed, the sequence $x$ is bounded since $T$ is power-bounded, and moreover 
$\smash{\sum_{k=0}^nx_n=I-T^{n+1}}$ for all $n\ge0,$ so \eqref{bdd} is also satisfied, again by power-boundedness of $T$. Furthermore, 
$G_x(\lambda)=(I-T)R(\lambda,T)$ for  $|\lambda|>1.$  Let $G_x$ denote also the extension of this map to the resolvent set $\rho(T)=\CC\backslash\sigma(T)$ and let $F_x$ be the restriction of $G_x$ to $\T\backslash\{1\}$. Note that  $\|R(\lambda,T)\|\ge \dist(\lambda,\sigma(T))^{-1}\ge |1-\lambda|^{-1}$  for all $\lambda\in\rho(T)$, and hence 
 \begin{equation}\label{lb}
 m(\varepsilon)\ge \frac{1}{|1-\e^{\ii\varepsilon}|}=\frac{1}{2\sin(\varepsilon/2)}\ge\frac{1}{\varepsilon},\quad 0<\varepsilon\le\pi.
 \end{equation}
 Further, for $k\ge0$,  
$$F_x^{(k)}(\lambda)=(-1)^k k! R(\lambda,T)^k\big(I+(1-\lambda)R(\lambda,T)\big),\quad \lambda\in\T\backslash\{1\},$$
and it follows from \eqref{lb} that \eqref{dom_fun} holds. Since $\mlog(\varepsilon)\gtrsim m(\varepsilon)\ge\varepsilon^{-1}$
 for all $\varepsilon\in(0,\pi]$, and therefore $n^{-1}\lesssim \mlog^{-1}(cn)$ for all sufficiently large $n\ge1$, \eqref{KT} follows from \eqref{bound}.
\end{proof}

\begin{rem}\label{KT_rem} 
\begin{enumerate}[(a)]
\item\label{opt_rem}  For functions $m:(0,\pi]\to[1,\infty)$ of the form $m(\varepsilon)=C\varepsilon^{-\alpha}$ for suitable constants $C>0$, as considered in the first part of Example~\ref{poly_ex}, the right-hand side in \eqref{KT} is given by that in \eqref{u_bd}.
It is shown in \cite[Section~3]{Se2} that the logarithmic factor in this expression can be dropped if $X$ is a Hilbert space but not for general Banach spaces. 
\item  It is possible to obtain a `local' version of Theorem~\ref{KT_quant} from Theorem~\ref{Ing} giving, for a fixed $x\in X$, an estimate for the rate of decay of $\|T^n(I-T)x\|$ as $n\to\infty$ which depends on the behaviour of the `local' resolvent operator $R(\lambda,T)x$ as $|\lambda|\to1+$; see for instance \cite{BNR98}, \cite{BV90}, \cite{BY00}, \cite{Ch98} and \cite{To01} for related local results in the context mainly of $C_0$-semigroups. Similarly, Theorem~\ref{Ing} can be used to obtain an estimate on the rate of decay of weak orbits $\phi(T^n(I-T)x)$ as  $n\to\infty$, where $x\in X$ and $\phi$ is a bounded linear functional on $X$.
\end{enumerate}
\end{rem}

\bibliography{/Users/David/Documents/TeX/Bibliographies/Seifert}
\bibliographystyle{plain}
\end{document}